\RequirePackage{plautopatch}
\documentclass[12pt]{amsart}
\usepackage{amssymb,amsthm,amsmath,bm,scalerel}
\usepackage{xcolor,hyperref}

\hypersetup{
    colorlinks=true,
    citecolor=blue,
    linkcolor=red,
    urlcolor=orange,
}

\newtheorem{thm}{Theorem}
\newtheorem{prop}{Proposition}

\def\PP{\mathcal P}
\def\QQ{\mathcal Q}
\def\R{\mathbb R}
\def\F{\mathbb F}
\def\xx{\bm{x}}
\def\yy{\bm{y}}
\def\zz{\bm{z}}
\def\vv{\bm{v}}
\def\burn{{\mathop{\mathrm{burn}}}}
\def\dist{{\mathop{\mathrm{dist}}}}

\usepackage{tikz}
\usetikzlibrary{positioning}
\usetikzlibrary{svg.path}
\definecolor{orcid_color}{HTML}{A6CE39}

\hypersetup{pdfborder={0 0 0}} 
\DeclareRobustCommand{\orcidicon}{%
	\raisebox{.2mm}{\scalerel*{%
	\begin{tikzpicture}[xscale=1,yscale=-1,transform shape]
	\filldraw[color=orcid_color] svg {M256,128c0,70.7-57.3,128-128,128C57.3,256,0,198.7,0,128C0,57.3,57.3,0,128,0C198.7,0,256,57.3,256,128z};
	\filldraw[color=white] svg {M86.3,186.2H70.9V79.1h15.4v48.4V186.2z} svg {M108.9,79.1h41.6c39.6,0,57,28.3,57,53.6c0,27.5-21.5,53.6-56.8,53.6h-41.8V79.1z M124.3,172.4h24.5
		c34.9,0,42.9-26.5,42.9-39.7c0-21.5-13.7-39.7-43.7-39.7h-23.7V172.4z} svg {M88.7,56.8c0,5.5-4.5,10.1-10.1,10.1c-5.6,0-10.1-4.6-10.1-10.1c0-5.6,4.5-10.1,10.1-10.1
		C84.2,46.7,88.7,51.3,88.7,56.8z};
	\end{tikzpicture}}{|}}%
}
\newcommand{\orcid}[1]{\href{https://orcid.org/#1}{\orcidicon}}

\newcommand{\arxiv}[1]{arXiv:\href{https://doi.org/10.48550/arXiv.#1}{#1}}

\begin{document}

\title[Burning numbers via eigenpolytopes]{Burning numbers via eigenpolytopes
\\--- Hamming graphs, Johnson graphs, and halved cubes}

\author[Hajime Tanaka]{Hajime Tanaka\,\orcid{0000-0002-5958-0375}}
\address{\href{https://www.math.is.tohoku.ac.jp/index.html}{Research Center for Pure and Applied Mathematics}, Graduate School of Information Sciences, Tohoku University, Sendai 980-8579, Japan}
\email{htanaka@tohoku.ac.jp}
\urladdr{https://hajimetanaka.org/}

\author[Norihide Tokushige]{Norihide Tokushige\,\orcid{0000-0002-9487-7545}}
\address{College of Education, University of the Ryukyus, Nishihara  903-0213, Japan}
\email{hide@cs.u-ryukyu.ac.jp}
\urladdr{http://www.cc.u-ryukyu.ac.jp/~hide/}

\begin{abstract}
We give lower and upper bounds on the burning number of Hamming graphs, Johnson
graphs, and halved cube graphs. For the lower bounds, we use the fact that $1$-skeletons of the
eigenpolytopes of these graphs are isomorphic to the original graphs. Then, we present
a dynamic search algorithm performed on the eigenpolytope to find an unburned 
vertex. This idea was originally used by Alon (Discrete Appl.\ Math.,\ 1992),
who determined the burning number of the hypercube graphs.
\end{abstract}

\maketitle

\section{Introduction}
We estimate or determine the burning number of Hamming graphs,
Johnson graphs, and halved cube graphs. These graphs have the remarkable
property that the $1$-skeletons of their eigenpolytopes are isomorphic to the
original graphs. Our proof uses the geometric structure of these eigenpolytopes.

We recall the burning number of a graph.
Let $X$ be a connected graph with vertex set $V$. 
The distance between two vertices $u,v\in V$, denoted by $\dist(u,v)$, is defined as the number of edges in a shortest path between them.
For a nonnegative integer $i$, we define the $i$-neighbor of a vertex $u$ by
$\Gamma_i(u):=\{v\in V:\dist(u,v)\leq i\}$, e.g., $\Gamma_0(u)=\{u\}$.
A sequence of $b$ vertices $v_1,v_2,\ldots,v_b$ satisfying
$\bigcup_{1\leq i\leq b}\Gamma_{b-i}(v_i)=V$ is called a \emph{burning sequence}
of length $b$. 
Suppose that we can choose one vertex at each discrete time step, and
if vertex $v_i$ is burned at time $i$, all vertices of $\Gamma_{b-i}(v_i)$ are burned by time $b$. Then, a burning sequence
of length $b$ tells us how to burn all the vertices of $X$ by time $b$.
The \emph{burning number} of $X$, denoted by $\burn(X)$, is the minimum number $b$ such that there exists a burning sequence of length $b$.
If $X$ has diameter $d$, then $\Gamma_d(v)=V$ for all $v\in V$. This yields
a trivial bound $\burn(X)\leq d+1$. See \cite{Bonato} for more on burning numbers.

Although computing the burning number is NP-complete (it is so even
when restricted to trees of maximum degree three; see \cite[Theorem 6]{Betal}), 
Alon~\cite{Alon} determined the burning number of the $n$-dimensional hypercube $Q_n$:
\begin{thm}[Alon]
$\burn(Q_n)=\lceil\frac n2\rceil+1$.
\end{thm}
\noindent
It is easy to show the upper bound. Indeed, take any vertex $v_1$, and let
$v_2$ be its antipodal vertex, then $\Gamma_{\lceil\frac n2\rceil}(v_1)\cup\Gamma_{\lceil\frac n2\rceil-1}(v_2)$ covers
all the $2^n$ vertices. It is, however, difficult to show the lower bound.
For this, for any given sequence of $\lceil\frac n2\rceil$ vertices, we need to show that it
cannot be a burning sequence, that is, we need to find a vertex that remains unburned at time $\lceil\frac n2\rceil$. 
To find such a vertex, Alon used a lemma due to Beck and Spencer~\cite{Beck-Spencer}, 
which was based on a dynamic search algorithm on the $n$-dimensional cube in $\R^n$. We will extend Alon's idea to graphs that can be realized as polytopes with some good properties, namely, \emph{eigenpolytopes}.

Now, we define an eigenpolytope for a graph $X$ with vertex set $V$.
Let $\theta$ be an eigenvalue of the adjacency matrix with multiplicity $m$.
Let $U_\theta$ be an $|V|\times m$ matrix whose columns consist of an orthonormal basis of the $\theta$-eigenspace. Write $u(\theta)$
for the row of $U_\theta$ that corresponds to $u\in V$. This defines an
embedding $\phi:V\to\R^m$ by $\phi(u)=u(\theta)$.
Then, we define the eigenpolytope of $X$ belonging to $\theta$, denoted by
$\PP_\theta$, by the convex hull of the set $\{u(\theta):u\in V\}$.
This definition depends on the choice of the orthonormal basis, but the
inner products $u(\theta)\cdot v(\theta)$ do not.
We identify polytopes modulo orthogonal transformations, and so the eigenpolytopes are well-defined. In this paper, we consider distance-regular graphs (see, e.g., \cite{Brouwer,vanDam}), and
we always assume that $\theta$ is the second largest eigenvalue.
In this case, Godsil~\cite[Theorem 4.3]{Godsil} established the following fundamental result.
\begin{thm}[Godsil]
Let $X$ be a distance-regular graph with second largest eigenvalue $\theta$.
Then, $X$ is isomorphic to the $1$-skeleton of the eigenpolytope $\PP_\theta$ 
if and only if it is one of the following:
(a) a Hamming graph $H(n,q)$,
(b) a Johnson graph $J(n,k)$,
(c) a halved $n$-cube $\frac12H(n,2)$,
(d) the Schl\"afli graph,
(e) the Gosset graph,
(f) the icosahedron,
(g) the dodecahedron,
(h) the complement of $r$ copies of $K_2$, or
(i) a cycle $C_n$.
\end{thm}

Among the graphs in the list, we are interested in (a), (b), and (c), because
the burning numbers of the other graphs are easily determined. Our main result
is the following.

\begin{thm}\label{thm:main}
For Hamming graphs, we have
\begin{equation}\label{eq:hamming}
\left\lfloor\left(1-\frac1q\right)n\right\rfloor+1\leq\burn(H(n,q))\leq\left\lfloor\left(1-\frac1q\right)n+\frac{q+1}2\right\rfloor.
\end{equation} 

For Johnson graphs, if $n\geq 2k$ then
\begin{equation}\label{eq:johnson}
\left\lceil\left(1-\frac kn\right)k\right\rceil+1\leq\burn(J(n,k))
\leq\left\lfloor\left(1-\frac1{\lceil\frac nk\rceil}\right)k+\frac12\left(\left\lceil\frac nk\right\rceil+1\right)\right\rfloor.
\end{equation}

For halved cubes, we have
\begin{align*}
\burn\big(\tfrac12H(n,2)\big)=\left\lceil \frac{n}{4}\right\rceil+1.  
\end{align*}
\end{thm}
\noindent
We will present a unified method that gives lower bounds on the burning number of these graphs. This method relies on the fact that all the points
of the eigenpolytope, corresponding to (a), (b), or (c), are on $d+1$ parallel hyperplanes arranged at the same interval, 
where $d$ is the diameter of the graph. 

Noting that $Q_n=H(n,2)$ and $\lfloor\frac n2+\frac32\rfloor=\lceil\frac n2\rceil+1$,
Alon's result shows that if $q=2$ then the correct value of the burning number in
\eqref{eq:hamming} coincides with the upper bound.
For $q\geq 3$, \eqref{eq:hamming} is obtained in \cite{T2025}, 
but our proof here is much simpler. The result \eqref{eq:johnson} looks complicated,
and we give two easy corollaries:
$\burn(J(2k,k))=\lceil\frac k2\rceil+1$, and if
$n>k^2$ then $\burn(J(n,k))=k+1$.
The burning number of (d), (e), (f), (g), (h), or (i) is
$3,3,3,4,2$, and $\lceil n^{1/2}\rceil$, respectively.
For (i), a shortest burning sequence $v_1,\dots,v_{\lceil n^{1/2}\rceil}$ is given by $v_j=(2\lceil n^{1/2}\rceil+1)j-j^2$, where
we assume $V(C_n)={\mathbb Z}_n$; see \cite[Corollary 11]{Bonatoetal}.

In the next section, we define distance-regular graphs appearing in
Theorem~\ref{thm:main} formally and gather their basic properties for our proof. In Section~\ref{sec:upper bounds}, we give burning sequences 
for the upper bounds of Theorem~\ref{thm:main}, where we only use graph structures. 
Finally, we prove the lower bounds of Theorem~\ref{thm:main} in 
Section~\ref{sec:lower bounds}, and this is where we use the fact that those
graphs are realized as the $1$-skeletons of their eigenpolytopes.

\section{Preliminaries}
\subsection{Cosines}
Let $X$ be a distance-regular graph with diameter $d$. 
For two vertices $u$ and $v$ of $X$ at distance $i$, we write $c_i,a_i$,
and $b_i$ for the number of neighbors of $v$ at distance $i-1,i$, and $i+1$
from $u$, respectively.
By the definition of distance-regularity, these numbers are independent of the pair $u,v$.
Let $w_i$ denote the cosine of the angle between the vectors
$u(\theta)$ and $v(\theta)$, where $\theta$ is the second largest eigenvalue.
We call $w_0,w_1,\ldots,w_d$ the \emph{cosine sequence} of $X$. 
Then, we have $w_0=1$, and
\begin{equation}\label{eq:3-term rec}
 w_{i+1}=\frac1{b_i}\left((\theta-a_i)w_i-c_iw_{i-1}\right) \qquad (i=0,1,\dots,d),
\end{equation}
where $w_{-1}$ and $w_{d+1}$ are indeterminates; see (3.2) of \cite{Godsil}. If, moreover, $X$ is one of Hamming graphs,
Johnson graphs, and halved cube graphs, then it follows from \eqref{eq:3-term rec} that the cosine sequence of $X$ is an
arithmetic progression, and 
\begin{equation}\label{eq:wi}
w_i=1-\left(1-\frac{\theta}{b_0}\right)i \qquad (i=0,1,\dots,d), 
\end{equation}
where $b_0$ is the valency (or degree) of $X$.
See also \cite[Corollary 8.4.2]{Brouwer}.

\subsection{Hamming graphs}
For integers $n\geq 1$ and $q\geq 2$,
the \emph{Hamming graph} $H(n,q)$  
has vertex set $V=\{(x_1,\ldots,x_n):x_i\in[q]\}$, where $[q]:=\{1,2,\ldots,q\}$, and
two vertices $x,y\in V$ are adjacent if and only if their Hamming distance
is one, that is, $\big|\{i:x_i\ne y_i\}\big|=1$; see \cite[Section 9.2]{Brouwer}. 
Then, $\dist(x,y)=\big|\{i:x_i\neq y_i\}\big|$, and
$H(n,q)$ is a $q^n$-vertex $n(q-1)$-regular graph of diameter $d=n$.
The second largest eigenvalue is $\theta=n(q-1)-q$ with multiplicity $m=n(q-1)$.
For $r_1,\ldots,r_n\subset [q]$ with $r_i\ne\emptyset$, let $K(r_1,\ldots,r_n)$ denote the direct product
of the complete graphs on $r_i$ for $i=1,2,\ldots,n$.
Then, $K(r_1,\ldots,r_n)$ is a convex subgraph of $H(n,q)$.
(We say that a subgraph $H$ of a graph $G$ is a \emph{convex subgraph} of $G$ if every shortest path in $G$ connecting two vertices of $H$ contains only vertices of $H$.)
Note that $H(n,q)=K([q],\ldots,[q])$.
Moreover, every convex subgraph of $H(n,q)$ is of this form, 
and there is a bijection between the set of all convex subgraphs and
the set of faces of the eigenpolytope $\PP_\theta$
(see \cite[Proposition~5.11]{Lambeck}, \cite[Theorem~7.1]{Mohri}).
For $0\leq i\leq m$,
the vertices of an $i$-face of $\PP_\theta$ induce a convex subgraph
$K(r_1,\ldots,r_n)$ with $\sum_{j=1}^n|r_j|=n+i$ of the $1$-skeleton $H(n,q)$, and this subgraph has diameter at most $i$. The cosine sequence is given by 
$w_j=1-\frac{qj}{n(q-1)}$ for $j=0,1,\ldots, n$.

\subsection{Johnson graphs}
For integers $n$ and $k$ with $1\leq k\leq n-1$,
the \emph{Johnson graph} $J(n,k)$ 
has vertex set
$V=\binom{[n]}k$, and two vertices $u,v\in V$ are adjacent if and only if $|u\cap v|=k-1$; see \cite[Section 9.1]{Brouwer}.
Then, $\dist(u,v)=|u\setminus v|=|v\setminus u|$, and
$J(n,k)$ is an $\binom nk$-vertex $k(n-k)$-regular graph of diameter $d=\min\{k,n-k\}$.
Since $J(n,k)$ is isomorphic to $J(n,n-k)$, we only consider the case $n\geq 2k$.
The second largest eigenvalue is $\theta=k(n-k)-n$ with multiplicity $m=n-1$.
For subsets $s$ and $t$ of $[n]$ such that $s\subset t$ and $|s|< k< |t|$, let
$F(s,t)$ denote the subgraph of $J(n,k)$ induced on $\big\{u\in\binom{[n]}k:s\subset u\subset t\big\}$.
Then, $F(s,t)$ is a convex subgraph of $J(n,k)$ and is isomorphic to $J(|t\setminus s|,k-|s|)$ with diameter $\min\{k-|s|,|t|-k\}$.
Moreover, every convex subgraph of $J(n,k)$ that is not a singleton is of this form, and there is a bijection between the set of all convex subgraphs and the set of faces of the eigenpolytope $\PP_\theta$ (see \cite[Proposition~5.7]{Lambeck}, \cite[p.~751]{Godsil}).
For $1\leq i\leq m$,
the vertices of an $i$-face of $\PP_\theta$ induce a convex subgraph $F(s,t)$ with $i=|t\setminus s|-1$ of the $1$-skeleton $J(n,k)$, and this subgraph has diameter at most $\lfloor \frac{i+1}{2}\rfloor$.
The cosine sequence is given by $w_j=1-\frac n{k(n-k)}j$ for 
$j=0,1,\ldots,k$.

\subsection{Halved cubes}
For an integer $n\geq 2$,
the \emph{halved $n$-cube} $\frac12 H(n,2)$ has vertex set 
$V=\{(x_1,\ldots,x_n):x_i\in\F_2,\,\sum_{i=1}^nx_i=0\}$, 
and two vertices $x,y\in V$ are adjacent if and only if their Hamming distance is two, that is, $\big|\{i:x_i\ne y_i\}\big|=2$; see \cite[Section 9.2D]{Brouwer}.
Then, $\dist(x,y)=\big|\{i:x_i\ne y_i\}\big|/2$, and
$\frac12 H(n,2)$ is a $2^{n-1}$-vertex $\binom n2$-regular graph of 
diameter $d=\lfloor n/2\rfloor$.
The second largest eigenvalue is $\theta=\frac{(n-1)(n-4)}{2}$ with multiplicity $m=n$ except when $n=2$, in which case we have $m=1$.
For $0\leq i\leq m$,
an $i$-face of $\PP_\theta$ is either an $i$-simplex $(i\geq 0)$ or $\frac12H(i,2)$ $(i\geq 4)$, 
and 
the corresponding diameter is either $1-\delta_{i0}$ or $\lfloor\frac{i}{2}\rfloor$
(see \cite[p.~752]{Godsil}).
The cosine sequence is given by $w_j=1-\frac4nj$ for $j=0,1,\ldots, d$.

\section{Upper bounds}\label{sec:upper bounds}
To show $\burn(X)\leq b$, it suffices to find a burning sequence
$v_1,\ldots,v_b\in V$, that is, $\bigcup_{1\leq i\leq b}\Gamma_{b-i}(v_i)=V$.
Sometimes, a shorter sequence will do.
For example, if $X=H(n,q)$ and $v_i=(i,i,\dots,i)$ for $i\in[q]$, then 
$\bigcup_{1\leq i\leq q}\Gamma_{b-i}(v_i)=V$ for some $b>q$.
Indeed, the following bound is shown.

\begin{prop}[{\cite[Theorem~2]{T2024}}]
$\burn(H(n,q))\leq\lfloor(1-\frac1q)n+\frac{q+1}2\rfloor$.
\end{prop}

\begin{prop}
Let $a=\lceil\frac nk\rceil$ and $b=(1-\frac1a)k+\frac {a+1}2$.
Then,
$\burn(J(n,k))\leq b$.
\end{prop}
\begin{proof}
Let $v_1=[k]$, $v_i=[ki]\setminus[k(i-1)]$ $(i=2,\dots,a-1)$, and let $v_a\in V$ be any vertex with $[n]\setminus[k(a-1)]\subset v_a$.
Since $[n]=v_1\cup v_2\cup\cdots\cup v_{a}$, it follows that
$k\leq\sum_{i=1}^{a}|u\cap v_i|$ for all $u\in V$.
We show that $\burn(J(n,k))\leq b$.
To this end, we burn $v_i$ at time $i$, and show that all vertices are
burned by time $\lfloor b\rfloor$. Suppose, to the contrary, that there is a vertex $x$
that remains unburned at time $\lfloor b\rfloor$.
Since $v_i$ is burned at time $i$, the vertices at distance within $\lfloor b\rfloor-i$ from $v_i$
are burned by time $\lfloor b\rfloor$.
By the assumption on $x$, we have 
\[
|v_i\setminus x|=\dist(v_i,x)> b-i \qquad (i=1,2,\dots,a),
\]
or equivalently, $|x\cap v_i|< k-b+i$ \,$(i=1,2,\dots,a)$.
It follows that
\[
 k\leq\sum_{i=1}^a|x\cap v_i|<\sum_{i=1}^a(k-b+i)
=a(k-b)+\frac{a(a+1)}2=k,
\]
a contradiction.
\end{proof}

\begin{prop}
$\burn(\frac12H(n,2))\leq \lceil \frac n4\rceil+1$.
\end{prop}
\begin{proof}
Let $v_1=(0,0,\dots,0)$, and let $v_2=(1,1,\dots,1)$ if $n$ is even, and
$v_2=(0,1,1,\dots,1)$ if $n$ is odd.
Let $b'=\lceil\frac n4\rceil$.
We show that if $v_1$ is burned at time $1$ and $v_2$ is burned at time $2$,
then all the vertices are burned by time $b'+1$, that is,
\begin{align}\label{eq:Gx1+Gx2=V}
\Gamma_{b'}(v_1)\cup\Gamma_{b'-1}(v_2)=V.
\end{align}
Let $d=\lfloor \frac n2\rfloor$ be the diameter of $\frac12H(n,2)$.
For $i=0,1,\dots, d$, let $V_i$ denote the set of vertices at distance $i$ from $v_1$.
Then, we have $V_0=\{v_1\}$, $V=V_0\cup V_1\cup\cdots\cup V_d$, and
$\Gamma_i(v_1)=V_0\cup V_1\cup\cdots\cup V_i$.

First, suppose that $n$ is even. Then, $V_d=\{v_2\}$,
$\Gamma_i(v_2)=V_d\cup V_{d-1}\cup\cdots\cup V_{d-i}$, and
\begin{align*}
\Gamma_{b'}(v_1)&=V_0\cup V_1\cup\cdots\cup V_{b'},\\
\Gamma_{b'-1}(v_2)&=V_{d}\cup V_{d-1}\cup\cdots\cup V_{d-b'+1}.
\end{align*}
Thus, \eqref{eq:Gx1+Gx2=V} holds if $(d-b'+1)-b'\leq 1$, that is,
$\lfloor\frac n2\rfloor\leq2\lceil\frac n4\rceil$,
which is certainly true.

Next, suppose that $n$ is odd. In this case, $V_d$ consists of the vectors
$(x_1,\ldots,x_n)\in\F_2^n$ such that $\big|\{i:x_i=0\}\big|=1$.
Then, $v_2\in V_d$, and
the subgraph of $\frac12H(n,2)$ induced on $V_d$ is isomorphic to the complete
graph $K_n$, from which it follows that $\Gamma_1(v_2)\supset V_d$, and
$\Gamma_i(v_2)\supset V_d\cup V_{d-1}\cup\cdots\cup V_{d-i+1}$. Then, we have
\begin{align*}
\Gamma_{b'}(v_1)&=V_0\cup V_1\cup\cdots\cup V_{b'},\\
\Gamma_{b'-1}(v_2)&\supset V_{d}\cup V_{d-1}\cup\cdots\cup V_{d-b'+2}.
\end{align*}
Then, \eqref{eq:Gx1+Gx2=V} holds if $(d-b'+2)-b'\leq 1$, that is,
$\lfloor\frac n2\rfloor\leq2\lceil\frac n4\rceil-1$,
which is certainly true.
\end{proof}

\section{Lower bounds}\label{sec:lower bounds}
Let $X_n$ be one of a Hamming graph $H(n,q)$, a Johnson graph $J(n,k)$, and
a halved $n$-cube $\frac12H(n,2)$.
Table~\ref{table1} below lists the corresponding
multiplicity $m$ of the second largest eigenvalue $\theta$, 
the diameter $d$, and the maximum diameter $D_i$ of the $1$-skeletons of the $i$-faces of the eigenpolytope. 
Let $b:=(1-\frac{\theta}{b_0})^{-1}$ so that the cosine sequence $w_0,w_1,\ldots,w_d$ is given by $w_i=1-\frac ib$ for $i=0,1,\dots,d$, see \eqref{eq:wi}.
Finally, we define an integer $b'$ as in the table.

\begin{table}[h]
\caption{}
\begin{center}
\begin{tabular}{c||c|c|c}
\label{table1}
& $H(n,q)$ & $J(n,k)$ & $\frac{1}{2}H(n,2)$ \rule[-3mm]{0mm}{5mm} \\
\hline
\hline
 $m$& $n(q-1)$& $n-1$& $n$ \rule[-2mm]{0mm}{7mm} \\
 $d$& $n$& $k$& $\lfloor\frac n2\rfloor$ \rule[-3mm]{0mm}{7mm} \\
$D_i$& $\min\{i,n\}$ & $\lfloor\frac{i+1}2\rfloor$ & $\max\big\{1-\delta_{i0},\lfloor\frac{i}{2}\rfloor\big\}$ \rule[-3mm]{0mm}{7mm} \\
 $b$& $(1-\frac1q)n$& $(1-\frac kn)k$& $\frac n4$ \rule[-3mm]{0mm}{7mm} \\
 $b'$& $\lfloor b\rfloor$& $\lceil b\rceil$& $\lceil b\rceil$ 
\end{tabular}
\end{center}
\end{table}

\begin{thm}
Let $X_n\in\big\{H(n,q),\,J(n,k),\,\frac12H(n,2)\big\}$. Then, we have $\burn(X_n)\geq b'+1$.
\end{thm}

\begin{proof}
Let $b'$ vertices $v_1,v_2,\ldots,v_{b'}$ of $X_n$ be given.
Suppose that $v_i$ is burned at time $i$ for each $i=1,2,\ldots,b'$. 
We need to find a vertex $x$ that remains unburned at time $b'$, that is,
\begin{align}\label{eq:d>b-i}
\dist(v_i,x)> b'-i
\end{align}
for every $i=1,2,\dots,b'$. 

It follows from $w_i=1-\frac ib$ that $i=b(1-w_i)$.
Let $\alpha_i$ be the cosine of the angle between the vectors $v_i(\theta)$ and
$x(\theta)$. Then, we have
\[
\dist(v_i,x)=b(1-\alpha_i).
\]
Hence, \eqref{eq:d>b-i} is equivalent to
\begin{align}\label{eq:alpha_i}
\alpha_i<\frac {i-b'+b}b 
\end{align}
for $i=1,2,\ldots,b'$. 
Note that the right hand side of \eqref{eq:alpha_i} is positive, and so
\eqref{eq:alpha_i} is true whenever $\alpha_i\leq 0$.

We also note that the vertices $v(\theta)$ $(v\in V)$ of $\PP_\theta$ have the same norm, and hence we can choose a positive constant $c$ so that all the vertices of the polytope $\QQ_\theta:=c\PP_\theta\subset\R^m$ are on the unit sphere centered at the origin.
(Specifically, we have $c=\sqrt{|V|/m}$.)
This defines an embedding $\psi:V\to\QQ_\theta$ by $\psi(v)=c\cdot v(\theta)$.

Choose additional vertices $v_{b'+1},v_{b'+2},\ldots,v_{m-1}$ arbitrarily, and let $\vv_i=\psi(v_i)$ for $i=1,2,\dots,m-1$.
We will find vectors
\begin{equation*}
    \xx_{m-1},\xx_{m-2},\ldots,\xx_1,\xx\in\R^m
\end{equation*}
in this order by the following procedure, 
where $\xx_{m-j}$ is on some $(m-j)$-face $F_{m-j}$ of $\QQ_\theta$.

\begin{description}
\item[Step $1$] Consider the system of equations
$\vv_i\cdot\zz=0$ $(i\in[m-1])$ with variables $\zz\in\R^{m}$. 
We have $m-1$ equations and $m$ variables, so the solution space contains a line 
passing through the origin.
Starting from the origin, we move along the line until we hit a facet of $\QQ_\theta$.
Let $F_{m-1}$ be the facet and let $\xx_{m-1}$ be the intersection point of the line
and the facet. (If there is more than one facet containing $\xx_{m-1}$, then 
choose one of the facets arbitrarily and let $F_{m-1}$ be the chosen facet.)
\end{description}
Next, we do the following for $j=2,3,\ldots,m-1$.
\begin{description}
\item[Step $j$] 
In the previous step $j-1$, we have found $\xx_{m-(j-1)}$ on $F_{m-(j-1)}$, 
which is an $(m-(j-1))$-face of $\QQ_\theta$. 

If $\xx_{m-(j-1)}$ is already on a facet of $F_{m-(j-1)}$, then 
let $\xx_{m-j}=\xx_{m-(j-1)}$ and let $F_{m-j}$ be the facet (or possibly one of 
the facets). Then, proceed to Step $j+1$.

Otherwise, solve the system of equations $\vv_i\cdot\zz=0$ $(i\in[m-j])$ with variables $\zz\in\R^{m}$.
The solution space contains a line in the affine hull of $F_{m-(j-1)}$ passing through $\xx_{m-(j-1)}$.
Starting from $\xx_{m-(j-1)}$, we move along the line until we hit a facet
of $F_{m-(j-1)}$. Let $F_{m-j}$ be the facet (or one of the facets), 
and let $\xx_{m-j}$ be the intersection point.
\end{description}
After finding $\xx_1$ and $F_1$, let $\xx=\xx_1$ if $\xx_1$ is already a vertex of $\QQ_\theta$,
otherwise let $\xx$ be one of the endpoints of the $1$-face 
$F_1$ so that $\vv_1\cdot\xx\leq \vv_1\cdot\xx'$, where $\xx'$ denotes the other endpoint.

Now, we focus on $\xx,\xx_1,\xx_2,\ldots,\xx_{b'}$ and $F_1,F_2,\ldots,F_{b'}$,
and will not use $\xx_j,F_j$ for $j>b'$.
Let $x\in V$ be the vertex of $X_n$ defined by $\psi(x)=\xx$. Then, $\alpha_i$
is the cosine of the angle between $\vv_i$ and $\xx$. We show that $x$ is the
desired vertex satisfying \eqref{eq:alpha_i} for $i=1,2,\ldots,b'$.

By construction, $\xx$ is a vertex of $\QQ_\theta$
and $\xx_i\in F_i$ $(i=1,2,\dots,b')$, where $F_i$ is an $i$-face with
$F_1\subset F_2\subset\cdots\subset F_{b'}$.
Note also that $\vv_j\cdot\xx_i=0$ for $j=1,2,\ldots,i$.
Write $H_0(\vv_i)$ for the hyperplane in $\R^m$ with normal vector $\vv_i$ passing through
the origin. Then, $\xx_i\in H_0(\vv_i)$. 
For $j=0,1,\dots,d$, write $H(\vv_i,w_j)$ for the hyperplane
parallel to $H_0(\vv_i)$ passing through $w_j\vv_i$.
Then, two consecutive hyperplanes $H(\vv_i,w_j)$ and $H(\vv_i,w_{j+1})$ are at distance $\frac{1}{b}$.
For fixed $i$, every vertex of $\QQ_\theta$ is on one of the $d+1$ hyperplanes
$H(\vv_i,w_j)$ ($j=0,1,\ldots,d$), and every edge ($1$-face) of $\QQ_\theta$
connects two vertices on the same hyperplane or two vertices from two consecutive hyperplanes.

By the choice of $\xx$, we see that $\alpha_1=\vv_1\cdot\xx\leq\vv_1\cdot\xx_1=0$. 
Thus, \eqref{eq:alpha_i} holds if $i=1$.
For the rest of the proof, let $2\leq i\leq b'$. We may assume that
$\alpha_i>0$, otherwise \eqref{eq:alpha_i} holds trivially.
Since $\xx\in F_i$ 
and $\alpha_1>0$, it follows that $F_i\not\subset H_0(\vv_i)$.
If $G:=F_i\cap H_0(\vv_i)$ is a face of $F_i$, then $\xx_i,\xx_{i-1},\dots,\xx_1,\xx\in G$ by construction, but then $\alpha_i=\vv_i\cdot\xx=0$, a contradiction.
Thus, $G$ is not a face of $F_i$,
and this means that the set $F_i\setminus H_0(\vv_i)$ is disconnected, from which it follows that there is a vertex $\yy$ of $F_i$
such that $\vv_i\cdot\yy<0$. 
\begin{figure}[h]
\centering
\begin{tikzpicture}[scale=.4]
\draw[line width=1] (5.3,4)--(4,4)--(0,0)--(10,0)--(14,4)--(7.2,4);
\draw[line width=1] (5.7,4)--(6.8,4);
\draw[line width=1] (4.8,9)--(4,9)--(0,5)--(10,5)--(14,9)--(7.2,9);
\draw[line width=1] (5.2,9)--(6.8,9);
\draw[line width=1] (6.8,14)--(4,14)--(0,10)--(10,10)--(14,14)--(7.2,14);
\draw[line width=1] (7,7)--(7,9.8);
\draw[line width=1, dotted] (7,10)--(7,12);
\draw[line width=1, dotted] (7,7)--(7,5);
\draw[line width=1] (7,2)--(7,4.8);
\draw[line width=1, dotted] (7,2)--(7,0);
\draw[line width=1] (7,-.2)--(7,-1);
\draw[line width=1, ->] (7,12)--(7,15);
\draw[line width=1, ->] (7,7)--(12,8);
\draw[line width=1] (7,7)--(4.2,9.8);
\draw[line width=1, dotted, ->] (4,10)--(3,11);
\draw[line width=1, dotted] (7,7)--(6,5);
\draw[line width=1, ->] (5.9,4.8)--(4,1);
\draw (0,2.5) node {$H_0(\bm{v}_i,\beta)$};
\draw (0,7.5) node {$H_0(\bm{v}_i)$};
\draw (0,12.6) node {$H_0(\bm{v}_i,\alpha_i)$};
\draw (8.1,12) node {$\alpha_i\bm{v}_i$};
\draw (7.9,2) node {$\beta\bm{v}_i$};
\draw (7.5,6.6) node {$\bm{0}$};
\draw (11.1,8.4) node {$\bm{x}_i$};
\draw (7.7,15) node {$\bm{v}_i$};
\draw (3.8,11) node {$\bm{x}$};
\draw (4.8,1) node {$\bm{y}$};
\end{tikzpicture}
\end{figure}
The $1$-skeleton of $F_i$ has diameter at most $D_i$
(as a graph), and hence we have $\dist(x,y)\leq D_i$, where $\psi(y)=\yy$. 
Write $\beta$ for the cosine of the angle between $\vv_i$ and $\yy$, that is, 
$\beta=\vv_i\cdot\yy<0$.
Since $\xx\in H(\vv_i,\alpha_i)$, 
$\yy\in H(\vv_i,\beta)$,
and these two hyperplanes are at distance $\alpha_i-\beta$, we have 
\[
 \alpha_i<\alpha_i-\beta\leq\frac{\dist(x,y)}b\leq\frac{D_i}b.
\]
Thus, to show \eqref{eq:alpha_i}, it suffices to show that $\frac{D_i}b\leq\frac{i-b'+b}b$.
If $X_n=H(n,q)$, then this is equivalent to $\lfloor b\rfloor\leq b$, which is certainly true.
If $X_n=J(n,k)$, then this is equivalent to
$\lceil b\rceil-b \leq i-\lfloor\frac{i+1}2\rfloor$.
The right-hand side is
increasing in $i$, and the inequality reads $\lceil b\rceil -b \leq 1$ for $i=2$, which is certainly true.
If $X_n=\frac12H(n,2)$, then this is equivalent to $\lceil b\rceil-b \leq i-\lfloor\frac{i}{2}\rfloor$ (for $i\geq 2$), which is similarly verified.
Consequently, the vertex $x\in V$ of $X_n$ with $x(\theta)=\xx$ satisfies \eqref{eq:alpha_i}, and the proof is complete.
\end{proof}

\section*{Acknowledgments}

HT was supported by JSPS KAKENHI Grant Number JP23K03064.
NT was supported by JSPS KAKENHI Grant Number JP23K03201.

\end{document}